\documentclass[11pt, a4paper]{article}
\usepackage{amsmath,amssymb}
\usepackage{bm}
\usepackage{float}
\usepackage{amsthm}
\newtheorem{theo}{Theorem}
\newtheorem{theorem}{Theorem}

\newtheorem{lemma}{Lemma}
\newtheorem{remark}{Remark}
\newtheorem{proposition}{Proposition}
\newtheorem{corollary}{Corollary}
\usepackage{enumerate}

  \bibliography{bibfile}

\title{The role of forward self-similar solutions in the Cauchy problem for semi-linear heat equations with exponential nonlinearity }
\author{Daesu Jeong  \\
Graduate School of Mathematics, Nagoya University \\
 JAPANESE
ADDRESS: Furocho, Chikusaku, Nagoya, Japan}
\date{\empty}
\begin{document}

\section*{Abstract}
We consider the Cauchy problem for semi-linear heat equations with exponential nonlinearity. The main purpose of this paper is to prove the existence of solutions lying on the borderline between global existence and blow-up infinite time. The existence has been shown for semi-linear heat equations with power type nonlinearity. We explain the main strategy to prove the existence. By using the definition of exponential function, we approximate the solution to exponential type equation by that of power type equation. Then we can use directly the knowledge for power type equation.
\\
\\
\maketitle

\section{Introduction}
In this paper, we consider the Cauchy problem: 
\begin{equation}\label{Cauchy problem}
\begin{cases}
u_t-\Delta u=e^u, &  (x,t)\in{\mathbb{R}^N \times (0,\infty)}, \\
u(x,0)=u_0(x), & x\in{\mathbb{R}^N},
\end{cases}
\end{equation}
where $N\geq1$ and $u_0$ is a continuous initial function. We will study the blow-up problem for (\ref{Cauchy problem}). We say that the solution $u$ to (\ref{Cauchy problem}) blows up in finite time if  there exists $T(u_0)<\infty$ such that $u\in{C^{2,1}(\mathbb{R}^N\times(0,T))\cap C(\mathbb{R}^N\times[0,T))}$ is a unique classical solution to (\ref{Cauchy problem}) which is finite in $\mathbb{R}^N\times[0,T(u_0))$ and satisfies 
$$
\limsup_{t\nearrow T(w_0)}\sup_{x\in{\mathbb{R}^N}}u(x,t)=+\infty.
$$
We say that $u$ is a local solution if $u\in{C^{2,1}(\mathbb{R}^N\times(0, \infty))\cap C(\mathbb{R}^N\times[0,\infty))}$ is a unique classical solution to (\ref{Cauchy problem}) which is finite in $\mathbb{R}^N\times[0,\infty)$. It is known that the initial function $u_0$ has to decay to $-\infty$ as $|x|\to\infty$ for the global solution to exist. Tello \cite{Tello} showed that problem (\ref{Cauchy problem}) has a global solution if there exist constants $\varepsilon\in{(0,2)}$ and $C>0$ such that 
\begin{equation}\label{assumption}
-Ce^{|x|^{2-\varepsilon}}\leq u_0(x)\leq C, \quad x\in{\mathbb{R}^N}.
\end{equation}
Throughout this paper, we assume the condition (\ref{assumption}). In this paper, we are interested in the existence of solution to (\ref{Cauchy problem}) lying on the borderline between global existence and blow-up in finite time.  \\
We introduce some known results for a semi-linear heat equation with power type nonlinearity. We consider the Cauchy problem:
\begin{equation}\label{Power type}
\begin{cases}
u_t-\Delta u = u^p, & (x,t)\in{\mathbb{R}^N\times(0,\infty)} \\
u(x,0)=u_0(x) & x\in{\mathbb{R}^N} 
\end{cases}
\end{equation}
where $u_t=\frac{\partial}{\partial t}u, \ \Delta u= \sum_{i=1}^{i=N}\frac{{\partial}^2}{{\partial x_i}^2} u, \ p>1$ and $u_0$ is a non-negative and bounded continuous initial function. It is well known that  the exponent $p_F:=(N+2)/N$ which is called the Fujita exponent, plays an important role in the existence of global solution of (\ref{Power type}). In fact, If $1<p\leq p_F$ then non-trivial non-negative solutions must blow-up in finite time. On the other hand, if $p>p_F$, there exist global solutions for suitable small initial data. The existence of global solution to problem (\ref{Power type}) strongly depends on the decay rate of initial function $u_0$ at $x=\infty$. In fact, Fujita \cite{Fujita} showed that (\ref{Power type}) has a global solution if  $u_0$ has the form of a small multiple of Gaussian, which decays exponentially at $x=\infty$. Weissler \cite{Weissler2} showed that (\ref{Power type}) has global solutions if $u_0$ has polynomial decay at $x=\infty$. Lee and Ni \cite{Lee} showed that the borderline decay rate of $u_0$ is to be $|x|^{-2/(p-1)}$ at $x=\infty$. In order to study the borderline decay rate, we consider the stationary problem of (\ref{Power type})  ,that is, positive solutions $u$ to the equation 
\begin{equation}\label{stationary problem of power type}
\Delta u + u^p=0 \quad \mbox{in} \ \mathbb{R}^N,
\end{equation}
where $N\geq3$. When $p>N/(N-2)$, equation (\ref{stationary problem of power type}) has a singular solution of the form:
$$
u^*(x):=l^*|x|^{-\frac{2}{p-1}}, \quad l^*:=\biggl(\frac{2}{p-1}\biggl(N-2-\frac{2}{p-1}\biggr)\biggr)^{1/(p-1)}.
$$
When $p\geq(N+2)/(N-2)$, equation (\ref{stationary problem of power type}) has one parameter family of radially symmetric regular solutions $\{u_{\alpha}\}_{\alpha}$ with initial condition $u_{\alpha}(0)=\alpha>0$, where every $u_{\alpha}$ satisfy $\lim_{|x|\to\infty}|x|^{\frac{2}{p-1}}u_{\alpha}(|x|)=L$ and their stability was studied in \cite{Gui}. Define the exponent $p_{JL}$ by 
$$
p_{JL}=\begin{cases} 
\infty, & 3\leq N\leq 10, \\
\displaystyle
1+\frac{4}{N-4-2\sqrt{N-1}}, & N\geq11.
\end{cases}
$$
This exponent $p_{JL}$ which is called the Joseph-Lundgren exponent plays an important role  in the stability of radially symmetric stationary solutions of (\ref{Power type}). \par
The equation in (\ref{Power type}) is invariant under the similarity transform 
$$ 
u_{\lambda}(x,t)={\lambda}^{2/(p-1)}u(\lambda x, \lambda^2t), \quad \mbox{for \ all} \ \lambda>0.
$$ 
In particular, a solution $u$ is said to be self-similar if 
\begin{equation}\label{self-similar solution}
u(x,t)={\lambda}^{2/(p-1)}u(\lambda x, \lambda^2t), \quad \mbox{for \ all} \ \lambda>0.
\end{equation}
We call the solution $u$ to (\ref{Power type}) the forward self-similar solution if $u$ is of the form:

\begin{equation}\label{forward self-similar solution}
u(x,t)=t^{-1/(p-1)}\varphi(x/\sqrt{t})
\end{equation}
where $\varphi$ satisfies the elliptic equation 
\begin{equation}\label{v}
\Delta \varphi+\frac{1}{2}x\cdot\nabla \varphi+\frac{1}{p-1}\varphi+\varphi^p=0 \ \mbox{in} \ \mathbb{R}^N.
\end{equation}
Such forward self-similar solutions are useful tools to describe the large time behavior of the solution to (\ref{Power type}). 
In particular, if $\varphi=\varphi(r), r=|x|,$ then $\varphi$ satisfies $\varphi'(0)=0$ and 
\begin{equation}\label{radial solution}
\varphi''+\biggl(\frac{N-1}{r}+\frac{r}{2}\biggr)\varphi'+\frac{1}{p-1}\varphi+\varphi^p=0 \quad \mbox{for} \ r>0.
\end{equation}
Then we can use ODE theory in investigating forward self-similar solutions.
We are interested in positive solutions $\varphi$ to (\ref{radial solution}) satisfying $\varphi'(0)=0$ and 
\begin{equation}\label{l}
\lim_{r\to\infty}r^{2/(p-1)}\varphi(r)=l
\end{equation}
with some $l>0$. For each $l>0$, we introduce the solution set 
\begin{equation}\label{S_l}
S_l=\{ \varphi\in{C^2[0,\infty)}: \varphi>0 \mbox{\ is a solution to (\ref{radial solution}) satisfying $\varphi'(0)=0$ and (\ref{l})}\}.
\end{equation}
We call $\underline{\varphi}_{l}$ a minimal solution of $S_l$ if $\underline{\varphi}_l\leq \varphi$ for all $\varphi\in{S_l}$. Naito \cite{Naito1} showed the existence of a minimal solution of $S_l$  by the comparison principle.
\begin{theo}[Naito \cite{Naito1}]\label{theoremA}
 Let $S_l$ be defined by \rm{(\ref{S_l})}. If $S_l\neq\emptyset,$ then $S_l$ has a minimal solution.
\end{theo}
Naito \cite{Naito3} also showed the following results.
\begin{theo}[Naito \cite{Naito3}]\label{theorem for power nonlinearity}
Let $p_F<p<p_{JL}.$ Assume that there exists a non minimal solution $\varphi_l$ of $S_l.$ Define a self-similar solution $u_l$ by 
\begin{equation}\label{thm-1}
{u}_l(x,t)=t^{-\frac{1}{(p-1)}}{\varphi}_l\biggl(\frac{|x|}{\sqrt{t}}\biggr).
\end{equation}

\begin{enumerate}
\item If $u_0(x)\geq u_l(x,t_0)$ and $u_0(x)\not\equiv u_l(x,t_0) \ \mbox{for} \ x\in{\mathbb{R}^N}$ with some $t_0>0$, then the solution $u$ to \rm{(\ref{Power type})}  blows up in finite time. 
\item If $u_0(x)\leq u_l(x,t_0)$ and $u_0(x)\not\equiv  u_l(x,t_0) \ \mbox{for} \ x\in{\mathbb{R}^N}$ with some $t_0>0$, then  the solution $u$ to \rm{(\ref{Power type})} exists globally in time. 
\end{enumerate}
\end{theo}

The purpose of this paper is to prove the same conclusions of Theorem \ref{theoremA} and B to problem (\ref{Cauchy problem}). We consider stationary solutions, that is, solutions to elliptic equation;
\begin{equation}\label{equation with exponential nonlinearity}
-\Delta u=e^{u}.
\end{equation}
For $N\geq3$, the function $u_*$ defined by
$$
u_*(x):=-2\log{|x|}+\log{(2N-4)},
$$
is a singular solution to problem (\ref{equation with exponential nonlinearity}). Fujishima \cite{Fujishima} showed that the decay rate $-2\log{|x|}$ at space infinity gives the critical decay rate for the existence of global solutions to (\ref{Cauchy problem}). In this paper we are concerned with the case where initial function $u_0$  decays to $-2\log{|x|}$ at space infinity, that is, 
$$
\lim_{|x|\to\infty}(2\log{|x|}+u_0(x))=L
$$ 
with $L\in{\mathbb{R}}$.
Note that the equation in (\ref{Cauchy problem}) does not have scale invariance, but the equation in (\ref{Cauchy problem}) is invariant under the transformation 
$$
u_{\lambda}(x,t)=\log{\lambda^2}+u(\lambda x, \lambda^2 t) \quad for \  \lambda>0.
$$
The function $u=u(x,t)$ is called a self-similar solution to the equation in (\ref{Cauchy problem}) if $u$ is of the form
\begin{equation}\label{self-similar with exp}
u(x,t)=-\log{t}+\varphi\biggl(\frac{x}{\sqrt{t}}\biggr),
\end{equation}
where $\varphi(y):=u(y,1)$ satisfies the elliptic equation
\begin{equation}\label{ee}
\Delta \varphi + \frac{1}{2} y\cdot\nabla \varphi+ e^{\varphi}+1= 0 \quad \mbox{in} \ \mathbb{R}^N.
\end{equation}
In particular, if $\varphi=\varphi(r),r=|y|,$ then $\varphi$ satisfies  
\begin{equation}\label{e}
\begin{cases} \displaystyle
\varphi''+\biggl(\frac{N-1}{r}+\frac{r}{2}\biggr)\varphi'+e^{\varphi}+1=0, & r>0, \\
\varphi'(0)=0 
\end{cases}
\end{equation}
We are interested in solutions $\varphi$ to (\ref{e}) satisfying 
\begin{equation}\label{L}
\lim_{r\to\infty}(2\log{r}+\varphi(r))=L
\end{equation}
with $L\in{\mathbb{R}}$. For any $L\in{\mathbb{R}}$, we introduce the solution set
\begin{equation}\label{S_L}
S_L:=\bigl\{\varphi\in{C^2([0,\infty))}: \varphi \ \mbox{is a solution to (\ref{e}) satisfying} \   (\ref{L}) \bigr\}.
\end{equation}
Then we are in position to state our main theorems:
\begin{theorem}\label{theorem5}
If $S_L\neq\emptyset,$ then there exists a minimal solution of $S_L$.
\end{theorem}
\begin{theorem}\label{main theorem}
Let $3\leq N\leq9.$ Assume that there exists a non-minimal solution $\varphi_L$ of $S_L$. Define a self-similar solution $u_L$ by
\begin{equation}\label{thm}
u_L(x,t)=-\log{t}+{\varphi}_L\biggl(\frac{|x|}{\sqrt{t}}\biggr).
\end{equation} 

\begin{enumerate}
\item If $u_0(x)\geq u_L(x,t_0)$ and $ u_0(x)\not\equiv u_L(x,t_0)$ for $x\in{\mathbb{R}^N}$ with some $t_0>0$, then the solution $u$ to \rm{(\ref{Cauchy problem})} blows up in finite time. 
\item If $u_0(x)\leq u_L(x,t_0)$ and $ u_0(x)\not\equiv u_L(x,t_0)$ for $x\in{\mathbb{R}^N}$ with some $t_0>0$, then the solution $u$ to \rm{(\ref{Cauchy problem})} exists globally in time.
\end{enumerate}
\end{theorem}
We remark that the assumption $p_{JL}=\infty$ when $3\leq N\leq10,$ here assumption $p_F<p<p_{JL}$ in Theorem B allows exponential nonlinearity in this case.  In the case $N=10,$ it is known by \cite{Fujishima} that there is no non-minimal solution of $S_L$ for any $L\in{\mathbb{R}}.$ \cite{Fujishima} also says that there exists an $L\in{\mathbb{R}}$ such that $S_L\neq\emptyset$ when $3\leq N\leq9.$ \\

We explain the main strategy to prove Theorem \ref{theorem5} and \ref{main theorem}. We first approximate the solution to equation (\ref{Cauchy problem}) by that of equation (\ref{Power type}) by using the formula 
$$
e^u=\lim_{n\to\infty}\biggl(1+\frac{u}{n}\biggr)^n;
$$
that is, we consider the following approximate equation
\begin{equation}\label{approximation}
u^{(n)}_t-\Delta u^{(n)}=\biggl(1+\frac{u^{(n)}}{n}\biggr)^n \quad \mbox{in} \ \mathbb{R}^N \times (0,\infty).
\end{equation}
Then we can use directly the knowledge for power type nonlinear equation (\ref{Power type}) to induce desired property for exponential type nonlinear equation (\ref{Cauchy problem}).
\\
The paper is organized as follows: In Section 2 we present some preliminary results. In Section 3 we prove the existence of approximate self-similar solution. In Section 4 we investigate properties of solution set $S_L$, in particular we establish the existence of a minimal solution of $S_L$ by using approximate solutions. In section 5, we prove Theorem \ref{main theorem}.
\section{The existence of approximate solutions.}
In this section we consider the non-linear heat equation:
\begin{equation}\label{approximation}
u^{(n)}_t-\Delta u^{(n)}=\biggl(1+\frac{u^{(n)}}{n}\biggr)^n \quad \mbox{in} \ \mathbb{R}^N \times (0,\infty).
\end{equation}
The equation in (\ref{approximation}) is invariant under the transformation:
$$
u^{(n)}_{\lambda}(x,t)=n(\lambda^{2/(n-1)}-1)+\lambda^{2/(n-1)}u^{(n)}(\lambda x, \lambda^2 t) \quad \mbox{for all} \ \lambda>0.
$$
In particular, we call $u^{(n)}$ a self-similar solution when $u^{(n)}=u^{(n)}_{\lambda}$ for all $\lambda>0.$ Forward self-similar solutions are of the form:
\begin{equation}\label{approximated forward self-similar solution}
u^{(n)}(x,t)=n(t^{-1/(n-1)}-1)+t^{-1/(n-1)}\varphi^{(n)}(\frac{x}{\sqrt{t}}),
\end{equation}
where $\varphi^{(n)}$ satisfies elliptic equation
\begin{equation*}
\Delta \varphi^{(n)} + \frac{1}{2} x\cdot\nabla \varphi^{(n)}+ \frac{1}{n-1}(\varphi^{(n)}+n)+\biggl(1+\frac{\varphi^{(n)}}{n}\biggr)^n= 0 \quad \mbox{in} \ \mathbb{R}^N.
\end{equation*}
Note here that $\varphi^{(n)}(r)$ of (\ref{approximated forward self-similar solution}) converges to $\varphi(r)$ of (\ref{self-similar with exp}) as $n\to\infty.$ 
In particular, if $\varphi^{(n)}=\varphi^{(n)}(r), r=|x|,$ then $\varphi^{(n)}$ satisfies 
\begin{equation}\label{n}
\begin{cases}\displaystyle
{\varphi^{(n)}}''+\biggl(\frac{N-1}{r}+\frac{r}{2}\biggr){\varphi^{(n)}}'+\frac{1}{n-1}(\varphi^{(n)}+n)+\biggl(1+\frac{\varphi^{(n)}}{n}\biggr)^n= 0, & r>0, \\
{\varphi^{(n)}}'(0)=0.
\end{cases} 
\end{equation}
We establish that the forward self-similar solution of semi-linear heat equations with exponential nonlinearity is approximated by that of semi-linear heat equations with power type nonlinearity.
\begin{theorem}\label{theorem3}
Let $\varphi_{\alpha}$ be the solution to \rm{(\ref{e})} with $\varphi_{\alpha}(0)=\alpha\in{\mathbb{R}}.$ Then there exists a sequence $\{\varphi^{(n)}_{\alpha}\}_{n\geq 1}$ of (\ref{n}) such that $\varphi^{(n)}_{\alpha}>-n$ and
\begin{equation}\label{sup}
\lim_{n\to\infty}\sup_{0\leq r \leq r_0}| \varphi^{(n)}_{\alpha}(r)-\varphi_{\alpha}(r)|=0 \quad \mbox{for} \  r_0>0.
\end{equation}
\end{theorem}

\begin{proof}[Proof of \rm{Theorem \ref{theorem3}}]
Let $n_0\in{\mathbb{N}}$ be chosen such that $n_0+\alpha>0.$ Let $\psi^{(n)}_{\alpha}(r)$ be the positive solution to the following differential equation:
\begin{equation}\label{va}
\begin{cases} \displaystyle
{\psi^{(n)}_{\alpha}}''+\biggl(\frac{N-1}{r}+\frac{r}{2}\biggr){\psi^{(n)}_{\alpha}}'+\frac{1}{n-1}\psi^{(n)}_{\alpha}+\biggl(\frac{\psi^{(n)}_{\alpha}}{n}\biggr)^n=0,  \quad   &n\geq n_0, \\
\psi^{(n)}_{\alpha}(0)=\alpha +n>0,  \quad {\psi^{(n)}_{\alpha}}'(0)=0,& n\geq n_0.
\end{cases}
\end{equation}
By (\ref{va}), $\psi^{(n)}_{\alpha}$ satisfies the following integral equations:
\begin{align}\label{eq1} 
\psi_{\alpha}^{(n)}(r)&=\alpha+n-\int_0^r\frac{1}{\rho_N(s)}\int_0^s\rho_N(t)\biggl[\frac{1}{n-1}\psi_{\alpha}^{(n)}(t)+\biggl(\frac{\psi_{\alpha}^{(n)}(t)}{n}\biggr)^n\biggr] \, dt \, ds, 
\\
\label{eq2}
{\psi^{(n)}_{\alpha}}'(r)&=-\frac{1}{\rho_N(r)}\int_0^r\rho_N(s)\biggr[\frac{1}{n-1}\psi_{\alpha}^{(n)}(s)+\biggl(\frac{\psi_{\alpha}^{(n)}(s)}{n}\biggr)^n\biggr] \, dt \, ds,
\end{align}
where $\rho_N(r)=r^{N-1}e^{\frac{r^2}{4}}.$ Since ${\psi^{(n)}_{\alpha}}'(r)<0$, we have 
\begin{equation}\label{eq3}
0<\psi_{\alpha}^{(n)}(r)\leq\alpha+n. 
\end{equation}
Put  $\varphi^{(n)}_{\alpha}=\psi^{(n)}_{\alpha}(r)-n$.  Since (\ref{eq1}), we have 
\begin{equation}\label{eq4} 
\varphi_{\alpha}^{(n)}(r)=\alpha-\int_0^r\frac{1}{\rho_N(s)}\int_0^s\rho_N(t)\biggl[\frac{1}{n-1}\psi_{\alpha}^{(n)}(t)+\biggl(\frac{\psi_{\alpha}^{(n)}(t)}{n}\biggr)^n\biggr] \, dt \, ds.
\end{equation}
We remark that $(1+a/n)^n\leq e^{a} \ (a>0).$ (\ref{eq3}) and (\ref{eq4}) imply that 
\begin{eqnarray}
|\varphi_{\alpha}^{(n)}(r)|&\leq& |\alpha| + \int_0^r\frac{1}{\rho_N(s)}\int_0^s\rho_N(t)\biggl[\frac{1}{n-1}(|\alpha|+n)+\biggl(1+\frac{|\alpha|}{n}\biggr)^n\biggr] \, dt \, ds, \nonumber\\
&\leq& |\alpha|+(e^{|\alpha|}+|\alpha|+2)\int_0^r\frac{1}{\rho_N(s)}\int_0^s\rho_N(t) \, dt \, ds, \nonumber \\
&\leq& |\alpha|+ (e^{|\alpha|}+|\alpha|+2)\int_0^r\int_0^s \, dt \, ds, \nonumber \\ 
\label{eq5}
&\leq& |\alpha|+ \frac{1}{2}(e^{|\alpha|}+|\alpha|+2)r_0^2,
\end{eqnarray}
 for all $r\in{[0,r_0]}.$ Thus we obtain that $\{\varphi_{\alpha}^{(n)}\}_{n\geq n_0}$  is uniformly bounded on $[0, r_0]$. From (\ref{eq2}) and (\ref{eq5}) , we see that 
\begin{eqnarray*}
|{\varphi^{(n)}_{\alpha}}'(r)|&=&|{\psi^{(n)}_{\alpha}}'(r)| \nonumber \\
&\leq& \frac{1}{\rho_N(r)}\int_0^r\rho_N(s)\biggl[\frac{1}{n-1}(|\alpha|+n)+\biggl(1+\frac{|\alpha|}{n}\biggr)^n\biggr] \, ds , \nonumber\\
&\leq& (e^{|\alpha|}+|\alpha|+2)r_0,
\end{eqnarray*}
for all $r\in{[0, r_0]}$. Thus we have deduced that $\{\varphi_{\alpha}^{(n)}\}_{n\geq n_0}$ is equi-continuous on $[0,r_0]$. By the Ascoli-Arzela theorem, there exists a subsequence of $\{\varphi_{\alpha}^{(n)}\}_{n\geq n_0}$ which converges to $\tilde{\varphi}_{\alpha}\in{C[0,r_0]}$ uniformly on $[0,r_0]$. Letting $n\to\infty$ in (\ref{eq4}) we have 
\begin{equation*}\label{eq6}
\tilde{\varphi}_{\alpha}(r)=\alpha-\int_0^r\frac{1}{\rho_N(s)}\int_0^s(1+e^{\tilde{\varphi}_{\alpha}(t)}) \, ds \,dt.
\end{equation*}
Thus $\tilde{\varphi}_{\alpha}\in{C^2}$ is the solution to (\ref{e}) with $\tilde{\varphi}_{\alpha}(0)=\alpha$ and ${\tilde{\varphi}_{\alpha}}'(0)=0.$ By the uniqueness of solution to ordinary differential equations, we conclude $\tilde{\varphi}_{\alpha}\equiv \varphi_{\alpha}.$ 
\end{proof}
The following theorem shows that $\varphi\in{S_L}$ is approximated by the solution $\varphi^{(n)}_{\alpha}$ with the aid of Theorem \ref{theorem3} . 
\begin{theorem}\label{theorem4}
 Let $\varphi_{\alpha}\in{S_L}$ with $\varphi_{\alpha}(0)=\alpha$. Assume that  $\{\varphi^{(n)}_{\alpha}\}_{n\geq 1}$ is given by \rm{Theorem \ref{theorem3}}. Then there exists $L^{(n)}(\alpha)\in{\mathbb{R}} \ (n\geq1)$ such that
\begin{equation}\label{L_n}
\lim_{r\to\infty}\bigl[r^{\frac{2}{n-1}}(\varphi^{(n)}_{\alpha}(r)+n)\bigr]-n=L^{(n)}(\alpha), \quad \lim_{n\to\infty}L^{(n)}(\alpha)=L.
\end{equation}    
\end{theorem}
\begin{remark}

Let $\psi^{(n)}$ be the solution to the equation:

\begin{equation}\label{psi}
\begin{cases} \displaystyle
{\psi^{(n)}}''+\biggl(\frac{N-1}{r}+\frac{r}{2}\biggr){\psi^{(n)}}'+\frac{1}{n-1}\psi^{(n)}+\biggl(\frac{\psi^{(n)}}{n}\biggr)^n=0   \\
{\psi^{(n)}}'(0)=0
\end{cases}
\end{equation}

For $l>0$, we are concerned with the solution set 
\begin{equation}\label{$S_{L^{(n)}+n}$}
S^{(n)}_L:=\{ \psi^{(n)}\in{C^2[0,\infty)}: \psi^{(n)}>0 \mbox{\ is a solution to \rm{(\ref{psi})} satisfying \\ 
$\displaystyle \lim_{r\to\infty}r^{\frac{2}{n-1}}\psi^{(n)}(r)=L$ }\}.
\end{equation}

Put $\psi^{(n)}_{\alpha}(r)=\varphi^{(n)}_{\alpha}+n$. Then $\psi^{(n)}_{\alpha}$ satisfies $\psi^{(n)}_{\alpha}>0$, \rm{(\ref{n})}, (\rm{\ref{sup}}),
\begin{equation*}\label{L_n}
\lim_{r\to\infty}\bigl[r^{\frac{2}{n-1}}\psi^{(n)}_{\alpha}(r)\bigr]=L^{(n)}(\alpha)+n, \quad and \quad \lim_{n\to\infty}L^{(n)}(\alpha)=L, 
\end{equation*} 
that is, $\psi^{(n)}_{\alpha}\in{S^{(n)}_{L^{(n)}(\alpha)+n}}$. 
\end{remark}
In order to prove Theorem \ref{theorem4},  we need the following proposition.
\begin{proposition}
 \label{lemma5'}
 Let $\psi^{(n)}=\psi^{(n)}_{\alpha}\in{C^2[0,\infty)} \ (n\geq 1)$ be the solution to \rm{(\ref{psi})} with $\psi^{(n)}_{\alpha}(0)=\alpha.$ Then there exists $C=C(\alpha)>0$ such that 
\begin{eqnarray}\label{eq7}
\biggl(\frac{|\psi^{(n)}(r)|}{n}\biggr)^n&\leq& C{(1+r)^{-2n/(n-1)}} \quad \mbox{for} \ r>0, \\
\label{eq78}
|{\psi^{(n)}}'(r)|&\leq& C(1+r)^{-2/(n-1)-1} \quad \mbox{for} \  r>0. 
\end{eqnarray}
We remark that Constant C do not depend on $n.$ To prove Proposition \ref{lemma5'}, we  introduce Energy function 
\begin{equation}\label{energy}
E^{(n)}(r)=\frac{{\psi^{(n)'}}^2(r)}{2}+\frac{1}{2(n-1)}{\psi^{(n)}}^2(r)+\frac{1}{n^n(n+1)}{\psi^{(n)}}^{n+1}(r), \quad r>0, n>1.
\end{equation}
\end{proposition}
Then, we prepare the following lemmas
\begin{lemma}
Let $\psi^{(n)}=\psi^{(n)}_{\alpha}\in{C^2[0,\infty)} \ (n\geq 1)$ be the solution to \rm{(\ref{psi})} with $\psi^{(n)}_{\alpha}(0)=\alpha.$ Assume that  $E^{(n)}(r)$ is given by (\ref{energy}). Then $E^{(n)}(r)$ is non increasing function in $r.$ In particular, $E^{(n)}(r)\leq E^{(n)}(0) \quad (r>0).$
\end{lemma}
\begin{proof}
\begin{eqnarray*}
\frac{d}{dr}E^{(n)}(r)&=&\biggl({\psi^{(n)}}''(r)+\frac{1}{n-1}\psi^{(n)}(r)+\biggl(\frac{\psi^{(n)}(r)}{n}\biggr)^n\biggr){\psi^{(n)}}' (r)\\
&=&-\biggl(\frac{N-1}{r}+\frac{r}{2}\biggr){{\psi^{(n)}}'}^2\leq0, 
\end{eqnarray*}
Thus $E^{(n)}(r)$ is non increasing in $r>0.$ In particular, $E^{(n)}(r)\leq E^{(n)}(0) \quad (r>0).$
\end{proof}
\begin{lemma}[\cite{Haraux} Proposition 3.1]
 \label{lemma5''}
 Let $\psi^{(n)}=\psi^{(n)}_{\alpha}\in{C^2[0,\infty)} \ (n\geq 1)$ be the solution to \rm{(\ref{psi})} with $\psi^{(n)}_{\alpha}(0)=\alpha.$ Then there exists $C=C(\alpha, n)>0$ such that 
\begin{eqnarray}\label{eq7}
|\psi^{(n)}(r)|&\leq& C(\alpha, n){(1+r)^{-2/(n-1)}} \quad \mbox{for} \ r>0, \\
\label{eq78}
|{\psi^{(n)}}'(r)|&\leq& C(\alpha, n)(1+r)^{-2/(n-1)-1} \quad \mbox{for} \  r>0. 
\end{eqnarray}
where $C(\alpha, n)=\sqrt{2(n-1)E^{(n)}(0)}.$ 
\end{lemma}
\begin{proof}[Proof of \rm{Proposition} \ref{lemma5'}]
By Lemma \ref{lemma5''}, we get the esitimates
\begin{eqnarray}\label{eq7}
|\psi^{(n)}(r)|&\leq& C(\alpha, n){(1+r)^{-2/(n-1)}} \quad \mbox{for} \ r>0, \\
\label{eq78}
|{\psi^{(n)}}'(r)|&\leq& C(\alpha, n)(1+r)^{-2/(n-1)-1} \quad \mbox{for} \  r>0. 
\end{eqnarray}
where $C(\alpha, n)=\sqrt{2(n-1)E^{(n)}(0)}.$ Since $\displaystyle \biggl(1+\frac{a}{n}\biggr)^n\leq e^a \quad (a>0),$ we have
\begin{eqnarray*}
\frac{1}{n}C(\alpha, n)&=&\frac{1}{n}\sqrt{(n-1)E^{(n)}(0)} \\
&=&\sqrt{\frac{(n-1)}{n^2}\biggl(\frac{1}{(n-1)}(\alpha+n)^2+\frac{2}{n^n(n+1)}(\alpha+n)^{n+1}\biggr)} \\
&\leq&\sqrt{\biggl(1+\frac{|\alpha|}{n}\biggr)^2+\frac{2}{n}\biggl(1+\frac{|\alpha|}{n}\biggr)^{n+1}} \\
&\leq&\biggl(1+\frac{|\alpha|}{n}\biggr)\sqrt{1+\frac{2}{n}\biggl(1+\frac{|\alpha|}{n}\biggr)^{n-1}} \\
&\leq&\biggl(1+\frac{|\alpha|}{n}\biggr)\sqrt{1+\frac{2}{n}e^{|\alpha|}}.
\end{eqnarray*}
 we obtain 
\begin{equation}\label{eqq1}
\biggl(\frac{1}{n}C(\alpha, n)\biggr)^n\leq e^{|\alpha|+e^{|\alpha|}}
\end{equation}
By (\ref{eq7}) and (\ref{eqq1}), we have
\begin{equation}\label{eqq2}
\biggl(\frac{|\psi^{(n)}(r)|}{n}\biggr)^n\leq C(\alpha)(1+r)^{-\frac{2n}{n-1}}
\end{equation}
Since (\ref{eqq2}) and  $\displaystyle  \lim_{n\to\infty}\biggl(\frac{(n+|\alpha|)^2}{(n-1)^2}+2\biggl(1+\frac{|\alpha|}{n}\biggr)^{n+1}\biggr)=1+2e^{|\alpha|},$ we get
\begin{eqnarray}
\frac{1}{n-1}C(\alpha, n)&\leq&\frac{1}{n-1}\sqrt{2(n-1)E^{(n)}(0)} \nonumber \\
&=&\sqrt{\frac{1}{n-1}\biggl(\frac{1}{(n-1)}(n+|\alpha|)^2+\frac{2}{n+1}\biggl(\frac{(n+|\alpha|)^{n+1}}{n^n}
\biggr)} \nonumber \\
&\leq&\sqrt{\frac{(n+|\alpha|)^2}{(n-1)^2}+2\biggl(1+\frac{|\alpha|}{n}\biggr)^{n+1}} \nonumber \\
&\leq&C(\alpha) \label{eqq3}
\end{eqnarray}
Since (\ref{eq7}), (\ref{eqq3}), we have
\begin{equation}\label{eqq4}
\frac{|\psi^{(n)}(r)|}{n-1}\leq C(\alpha)(1+r)^{-\frac{2}{n-1}} 
\end{equation}
By (\ref{eq2}),  (\ref{eqq2}) and (\ref{eqq4}) we have
\begin{eqnarray}
|{\psi^{(n)}}'(r)|&\leq&r^{1-N}e^{-\frac{r^2}{4}}\int_0^rs^{N-1}e^{\frac{s^2}{4}}\biggr[\frac{1}{n-1}|\psi^{(n)}(s)|+\biggl(\frac{|\psi^{(n)}(s)|}{n}\biggr)^n\biggr] \, ds \nonumber \\
&\leq&{C(\alpha)}e^{-\frac{r^2}{4}}\int_0^re^{\frac{s^2}{4}}\biggr[(1+s)^{-\frac{2}{n-1}} + (1+s)^{-\frac{2n}{n-1}}\biggr] \, ds\nonumber \\
&\leq&{C(\alpha)}e^{-\frac{r^2}{4}}\int_0^re^{\frac{s^2}{4}}(1+s)^{-\frac{2}{n-1}} \, ds \nonumber \\
&\leq&C(\alpha)e^{-\frac{r^2}{4}}\biggl(\int_0^{\frac{r}{2}}e^{\frac{s^2}{4}} \, ds+\int_{\frac{r}{2}}^re^{\frac{s^2}{4}}(1+s)^{-\frac{2}{n-1}}\, ds\biggr) \nonumber \\
&\leq&C(\alpha)\biggl[e^{-\frac{3r^2}{16}}+(1+\frac{r}{2})^{-\frac{2}{n-1}-1}e^{-\frac{r^2}{4}}\int_{\frac{r}{2}}^r(1+s)e^{\frac{s^2}{4}}\, ds\biggr] \label{eqq4}
\end{eqnarray}
If $r<2,$  Right hand side of (\ref{eqq4}) is bounded. If $r\geq2,$ Since
$$
\int_{\frac{r}{2}}^r2se^{\frac{s^2}{4}}\, ds=4e^{\frac{r^2}{4}}-4e^{\frac{r^2}{16}}\leq4e^{\frac{r^2}{4}}, 
$$
Right hand side of (\ref{eqq4}) is bounded. Therefore we obtain 
$$
|{\psi^{(n)}}'(r)|\leq C(\alpha)(1+r)^{-\frac{2}{n-1}-1}.
$$
\end{proof}
\begin{lemma}\label{lemma33}
Let $\varphi\in{C^2[0,\infty)}$ be the solution to \rm{(\ref{e})} with $\varphi(0)=\alpha.$ Then there exists a constant $C=C(\alpha)>0$ such that
$$ 
|\varphi'(r)|\leq C(1+r)^{-1} \quad for \ r>0.
$$
\end{lemma}
\begin{proof}[Proof of \rm{Theorem \ref{theorem4}}]
The following argument can be found in the proof of [\cite{Haraux} proposition 3.4].
From Theorem \ref{theorem3}, there exists a sequence $\{\varphi^{(n)}_{\alpha}\}_{n\geq 1}$ of (\ref{n}) such that $\varphi^{(n)}_{\alpha}=\varphi^{(n)}>-n$ and (\ref{sup}). Put $\psi^{(n)}=\varphi^{(n)}+n$. Then $\psi^{(n)}$ satisfies (\ref{va}). The identity 
$$
(r^{2/(n-1)}\psi^{(n)})'=r^{2/(n-1)-1}(r{\psi^{(n)}}'+\frac{2}{n-1}\psi^{(n)})
$$  
and (\ref{va}) implies that
\begin{equation}\label{eq9}
\begin{split}
&\frac{d}{dr}\biggl[r^{2/(n-1)}\psi^{(n)}(r)+2r^{2/(n-1)-1}{\psi^{(n)}}'(r)\biggr]  \\
&= 2(\frac{2}{n-1}-N)r^{2/(n-1)-2}{\psi^{(n)}}'(r)-2r^{2/(n-1)-1}\biggl(\frac{\psi^{(n)}(r)}{n}\biggr)^n. \\
\end{split}
\end{equation}
Integrating (\ref{eq9}) from $1$ to $r$, we have
\begin{equation}\label{eq10}
\begin{split}
&r^{2/(n-1)}\psi^{(n)}(r)+2r^{2/(n-1)-1}{\psi^{(n)}}'(r)-\psi^{(n)}(1)-2{\psi^{(n)}}'(1)  \\
&= 2(\frac{2}{n-1}-N)\int_1^r t^{2/(n-1)-2}{\psi^{(n)}}'(t) \, dt-2\int_1^rt^{2/(n-1)-1}\biggl(\frac{\psi^{(n)}(t)}{n}\biggr)^n \, dt. 
\end{split}
\end{equation}
Note that we have
$$
\int_1^{\infty} t^{2/(n-1)-2}{\psi^{(n)}}'(t) \, dt<\infty  \quad \mbox{and} \quad
\int_1^{\infty}t^{2/(n-1)-1}\biggl(\frac{\psi^{(n)}(t)}{n}\biggr)^n \, dt<\infty,
$$
by Proposition \ref{lemma5'}. Letting $r\to\infty$ in (\ref{eq10}), we get
\begin{equation}\label{e11}
\begin{split}
&\lim_{r\to\infty}(r^{2/(n-1)}\psi^{(n)}(r))-\psi^{(n)}(1)-2{\psi^{(n)}}'(1)  \\
&= 2(\frac{2}{n-1}-N)\int_1^{\infty} t^{2/(n-1)-2}{\psi^{(n)}}'(t) \, dt-2\int_1^{\infty}t^{2/(n-1)-1}\biggl(\frac{\psi^{(n)}(t)}{n}\biggr)^n \, dt. 
\end{split}
\end{equation}
Since $\psi^{(n)}=\varphi^{(n)}+n,$ we obtain 
\begin{equation}\label{e12}
\begin{split}
&L^{(n)}({\alpha})-\varphi^{(n)}(1)-2{\varphi^{(n)}}'(1)  \\
&= 2(\frac{2}{n-1}-N)\int_1^{\infty} t^{2/(n-1)-2}{\psi^{(n)}}'(t) \, dt-2\int_1^{\infty}t^{2/(n-1)-1}\biggl(\frac{\psi^{(n)}(t)}{n}\biggr)^n \, dt. 
\end{split}
\end{equation}
By Proposition \ref{lemma5'}, there exists a constant $C>0$ such that 
\begin{eqnarray}\label{e12*}
\big|t^{2/(n-1)-2}{\psi^{(n)}}'(t)\big|&\leq& C(1+t)^{-3}, \\
\bigg|t^{2/(n-1)-1}\biggl(\frac{\psi^{(n)}(t)}{n}\biggr)^n\bigg|&\leq& C(1+t)^{-3}
\end{eqnarray}
and we have
\begin{eqnarray*}\label{e12**}
\lim_{n\to\infty}t^{2/(n-1)-2}{\psi^{(n)}}'(t)&=&\lim_{n\to\infty}t^{2/(n-1)-2}\frac{d}{dt}[{\varphi^{(n)}}(t)+n] \\
&=&\lim_{n\to\infty}t^{2/(n-1)-2}{\varphi^{(n)}}'(t)\\
&=&t^{-2}{\varphi}'(t),  \quad t\in{\mathbb{R}}, \\ 
\lim_{n\to\infty}t^{2/(n-1)-1}\biggl(\frac{\psi^{(n)}(t)}{n}\biggr)^n &=& \lim_{n\to\infty}t^{2/(n-1)-1}\biggl(1+\frac{\varphi^{(n)}(t)}{n}\biggr)^n \\
&=&t^{-1}e^{\varphi(t)}, \quad t\in{\mathbb{R}}.
\end{eqnarray*}
Letting $n\to\infty$ in (\ref{e12}), we have
\begin{equation}\label{e13}
\begin{split}
&\lim_{n\to\infty}L^{(n)}({\alpha})-\varphi(1)-2{\varphi}'(1)  \\
&=-2N\int_1^{\infty} t^{-2}{\varphi}'(t) \, dt-2\int_1^{\infty}t^{-1}e^{\varphi} \, dt,
\end{split}
\end{equation}
by the Lebesgue convergence theorem and (\ref{sup}). Thus $\lim_{n\to\infty}L^{(n)}({\alpha})$ exists. On the other hand, since
$$
(2\log{r}+\varphi(r))'=r^{-1}(r\varphi'(r)+2),
$$
we have
\begin{equation}\label{e14}
\frac{d}{dr}(2\log{r}+\varphi(r)+2r^{-1}\varphi'(r))
= -2Nr^{-2}\varphi'(r)-2r^{-1}e^{\varphi(r)}.
\end{equation}
We remark that $\varphi'(r)/r\to0 \ \mbox{as} \ r\to\infty$ by Lemma \ref{lemma33}. Integrating (\ref{e14}) from $1$ to $\infty$, we have  
\begin{equation}\label{e15}
L-\varphi(1)-2\varphi'(1)= -2N\int_{1}^{\infty}t^{-2}\varphi'(t) \, dt-2\int_{1}^{\infty}t^{-1}e^{\varphi(t)} \, dt.
\end{equation}
From (\ref{e13}) and (\ref{e15}), we conclude that
$$
\lim_{n\to\infty}L^{(n)}({\alpha})=L.
$$
\end{proof}

\section{Properties  of solution set $S_L$.}

In this section, we will demonstrate the existence of a minimal solution of solution set $S_L$. To prove Theorem \ref{theorem5}, we prepare the following lemma.
\begin{lemma}[\cite{Naito1}  Lemma 3.1]\label{lemma3}
Let $S_{l}^{(n)}$ be defined by \rm{(\ref{$S_{L^{(n)}+n}$})}. If $S_{l}^{(n)}\neq\emptyset,$ then $S_{l}^{(n)}$ has a minimal solution.
\end{lemma}

\begin{proof}[Proof of \rm{Theorem 1}]  Let $\varphi\in{S_L}$ with $\varphi(0)=\alpha$. Assume that $\varphi^{(n)}=\varphi^{(n)}_{\alpha}$ and let $L^{(n)
 }=L^{(n)}(\alpha)$ be defined by \rm{Theorem \ref{theorem4}}. $\psi^{(n)}=\varphi^{(n)}+n$. Take $n\in{\mathbb{N}}$ so large that $L^{(n)}+n>0$. Then we have
$$
\lim_{r\to\infty}r^{2/(n-1)}\psi^{(n)}(r)>0,
$$
that is, $\psi^{(n)}\in{S_{L^{(n)
 }+n}^{(n)}}.$ Hence there exists a minimal solution $\underline{\psi}^{(n)}\in{S_{L^{(n)
 }+n}^{(n)}}$  by Lemma \ref{lemma3}. We remark that $\underline{\psi}^{(n)}$  does not depend on $\varphi(0)=\alpha.$ Since  $\underline{\psi}^{(n)}$ satisfies (\ref{va}), we have the following integral equations:
\begin{eqnarray*}\label{eq87} 
\underline{\psi}^{(n)}(r)&=&\underline{\psi}^{(n)}(0)-\int_0^r\frac{1}{\rho_N(s)}\int_0^s\rho_N(t)\biggl[\frac{1}{n-1}\underline{\psi}^{(n)}(t)+\biggl(\frac{\underline{\psi}^{(n)}(t)}{n}\biggr)^n\biggr] \, dt \, ds, 
\\
\label{eq88}
{\underline{\psi}^{(n)'}}(r)&=&-\frac{1}{\rho_N(r)}\int_0^r\rho_N(s)\biggr[\frac{1}{n-1}\underline{\psi}^{(n)}(s)+\biggl(\frac{\underline{\psi}^{(n)}(s)}{n}\biggr)^n\biggr] \, ds,
\end{eqnarray*}
where $\rho_N(r)=r^{N-1}e^{r^2/4}$.
Put $\underline{\varphi}^{(n)}=\underline{\psi}^{(n)}-n.$
Since $\underline{\psi}^{(n)}(r)\leq{\psi}^{(n)}(r) \ (r>0),$ we have $\underline{\varphi}^{(n)}(0)\leq\alpha.$
We now claim that $\{\underline{\varphi}^{(n)}(0)\}$ is bounded below. We integrate equation (\ref{eq9}) with $\psi^{(n)}$ replaced by $\underline{\psi}^{(n)}$ from $1$ to $r.$ Then 
\begin{equation}\label{1}
\begin{split}
&r^{2/(n-1)}\underline{\psi}^{(n)}(r)+2r^{2/(n-1)-1}{\underline{\psi}^{(n)}}'(r)-\underline{\psi}^{(n)}(1)-2{\underline{\psi}^{(n)}}'(1) \\
&= 2(\frac{2}{n-1}-N)\int_1^r t^{2/(n-1)-2}{\underline{\psi}^{(n)}}'(t) \, dt-2\int_1^rt^{2/(n-1)-1}\biggl(\frac{\underline{\psi}^{(n)}(t)}{n}\biggr)^n \, dt, 
\end{split}
\end{equation}
since $\displaystyle \lim_{r\to0}r^{2/(n-1)}\underline{\varphi}^{(n)}(r)=0,  \ 2\lim_{r\to0}r^{2/(n-1)-1}{\underline{\varphi}^{(n)}}'(r)=0.$
Letting $r\to\infty$ in (\ref{1}), we have
\begin{equation}\label{e11}
\begin{split}
&L^{(n)}+n-\underline{\psi}^{(n)}(1)-2{\underline{\psi}^{(n)}}'(1) \\
 &= 2(\frac{2}{n-1}-N)\int_1^{\infty} t^{2/(n-1)-2}{\underline{\psi}^{(n)}}'(t) \, dt-2\int_1^{\infty}t^{2/(n-1)-1}\biggl(\frac{\underline{\psi}^{(n)}(t)}{n}\biggr)^n \, dt. 
\end{split}
\end{equation}
By $\lim_{n\to\infty}L^{(n)}=L,$ there exists $C>0$ such that 
\begin{equation}\label{eqq5}
|L^{(n)}|\leq C \quad \mbox{for} \ n\in{\mathbb{N}}.
\end{equation}
Since Proposition \ref{lemma5'} and (\ref{eqq5}), we have
\begin{eqnarray*}
|\underline{\psi}^{(n)}(1)-n|&\leq&|L^{(n)}|+{|\underline{\psi}^{(n)}}'(1)| \\
&&  \hspace{-3cm}+2(\frac{2}{n-1}-N)\int_1^\infty t^{2/(n-1)-2}|{\underline{\psi}^{(n)}}'(t)| \, dt+2\int_1^{\infty}t^{2/(n-1)-1}\biggl(\frac{|\underline{\psi}^{(n)}(t)|}{n}\biggr)^n \, dt \\
&\leq&2C+2C\int_1^\infty t^{-3} \, dt \\
&\leq&C
\end{eqnarray*}
Thus $\{\underline{\psi}^{(n)}(1)-n\}_{n\in{\mathbb{N}}}$ is bounded. Then there exists $C>0$ such that $|\underline{\psi}^{(n)}(1)-n|\leq C.$ Since $\psi^{(n)}(r)$  is non increasing in $r>0,$ we obtain 
$$
-C\leq\underline{\psi}^{(n)}(1)-n\leq\underline{\psi}^{(n)}(0)-n
$$
Therefore $\{\underline{\psi}^{(n)}(0)-n\}_{n\in{\mathbb{N}}}$ is bounded. By the Bolzano-Weierstrass theorem, there exists a subsequence $\underline{\varphi}^{(n_k)}(0)$ of $\underline{\varphi}^{(n)}(0).$ Then $\underline{\varphi}^{(n_k)}$ satisfies the following:
\begin{eqnarray*}\label{eq89} 
\underline{\varphi}^{(n_k)}(r)&=&\underline{\varphi}^{(n_k)}(0)-\int_0^r\frac{1}{\rho_N(s)}\int_0^s\rho_N(t)\biggl[\frac{1}{n-1}(\underline{\varphi}^{(n_k)}(t)+n)+\biggl(1+\frac{\underline{\varphi}^{(n_k)}(t)}{n}\biggr)^n\biggr] \, dt \, ds, 
\\
\label{eq90}
{\underline{\varphi}^{(n_k)}}'(r)&=&-\frac{1}{\rho_N(r)}\int_0^r\rho_N(s)\biggr[\frac{1}{n-1}(\underline{\varphi}^{(n_k)}(s)+n)+\biggl(1+\frac{\underline{\varphi}^{(n_k)}(s)}{n}\biggr)^n\biggr] \, ds,
\end{eqnarray*}
where $\rho_N(r)=r^{N-1}e^{\frac{r^2}{4}}.$ By the same argument as that in Theorem \ref{theorem3}, $\underline{\varphi}^{(n_k)}$ converges to some $\underline{\varphi}$  uniformly in ${[0,r_0]}$. In particular, $\underline{\varphi}^{(n_k)}$ converges pointwisely to $\underline{\varphi}.$ We show that  $\displaystyle \lim_{r\to\infty}(\underline{\varphi}(r)+2\log{r})=L.$ since (\ref{1}), we have
\begin{equation}
\begin{split}\label{eqq6}
&r^{2/(n-1)}\underline{\varphi}^{(n)}(r)+nr^{2/(n-1)}-n+2r^{2/(n-1)-1}{\underline{\varphi}^{(n)}}'(r)-\underline{\varphi}^{(n)}(1)-2{\underline{\varphi}^{(n)}}'(1) \\
&= 2(\frac{2}{n-1}-N)\int_1^r t^{2/(n-1)-2}{\underline{\varphi}^{(n)}}'(t) \, dt-2\int_1^rt^{2/(n-1)-1}\biggl(1+\frac{\underline{\varphi}^{(n)}(t)}{n}\biggr)^n \, dt
\end{split}
\end{equation}
Letting $n\to\infty$  in (\ref{eqq6}), we have 
\begin{equation}
\begin{split}\label{eqq7}
\underline{\varphi}(r)+2\log{r}+2r^{-1}{\underline{\varphi}}'(r)-&\underline{\varphi}(1)-2{\underline{\varphi}}'(1) \\
&= -2N\int_1^r t^{-2}{\underline{\varphi}}'(t) \, dt-2\int_1^rt^{-1}e^{\underline{\varphi}(t)} \, dt, 
\end{split}
\end{equation}
for $r>0.$ Letting $r\to\infty$ in (\ref{eqq7}), we obtain 
\begin{equation}
\begin{split}\label{eqq8}
\lim_{r\to\infty}\bigl(\underline{\varphi}(r)+2\log{r}\bigr)-&\underline{\varphi}(1)-2{\underline{\varphi}}'(1) \\
&= -2N\int_1^{\infty} t^{-2}{\underline{\varphi}}'(t) \, dt-2\int_1^{\infty}t^{-1}e^{\underline{\varphi}(t)} \, dt, 
\end{split}
\end{equation}
by $\displaystyle \lim_{r\to\infty}r^{-1}{\underline{\varphi}}'(r)=0.$ On the other hand, Letting $n\to\infty$ in (\ref{e15}) with $\varphi$   replaced by $\underline{\varphi},$ we obtain 
\begin{equation}\label{eqq9}
L-\underline{\varphi}(1)-2\underline{\varphi}'(1)= -2N\int_{1}^{\infty}t^{-2}\underline{\varphi}'(t) \, dt-2\int_{1}^{\infty}t^{-1}e^{\underline{\varphi}(t)} \, dt.
\end{equation}
From (\ref{eqq8}) and (\ref{eqq9}), we obtain  $\displaystyle \lim_{r\to\infty}(\underline{\varphi}(r)+2\log{r})=L.$  Therefore $\underline{\varphi}\in{S_L}.$
 From $\underline{\varphi}^{(n_k)}\leq \varphi^{(n_k)}$, letting $n_k\to\infty$, we conclude that $\underline{\varphi}\leq \varphi$.  Note that $\underline{\varphi}$ does not depend on  $\varphi.$ Therefore $\underline{\varphi}$ is a minimal solution of $S_L$ ,i.e., $\underline{\varphi}\leq\varphi \ \mbox{for all} \ \varphi\in{S_L}.$
\end{proof}

\begin{corollary}\label{Cor1}
Assume that there exist at least two solutions $\underline{\varphi}$ and $\varphi$ of $S_L,$ where $\underline{\varphi}$ is a minimal solution of $S_L.$ Then there exist at least two solutions $\underline{\psi}^{(n)}$ and ${\psi}^{(n)}$ of $S^{(n)}_{L^{(n)}+n},$ where $\underline{\psi}^{(n)}$ is a minimal solution of $S^{(n)}_{L^{(n)}+n}.$
\end{corollary}

\begin{proof}
In the proof of Theorem \ref{theorem5}, there exist $\underline{\psi}^{(n)}$ and  ${\psi}^{(n)}$ of  $S^{(n)}_{L^{(n)}+n}$ such that $\underline{\varphi}^{(n)}:=\underline{\psi}^{(n)}-n$ and $\varphi^{(n)}:={\psi}^{(n)}-n$ converge to $\underline{\varphi}$ and $\varphi$, respectively, where $\underline{\psi}^{(n)}$ is a minimal solution of $S^{(n)}_{L^{(n)}+n}.$ 
\end{proof}

We will show the following properties of $S_L$. 

\begin{proposition}\label{prop7}
Let $S_L$ be given by \rm{(\ref{S_L})}. Assume that there exist at least two solutions $\underline{\varphi}_L$ and $\varphi_L$ of $S_L,$ where $\underline{\varphi}_L$ is a minimal solution of $S_L$. 
\begin{enumerate}
\item If $\varphi\in{S_L}$ satisfies $\varphi(r)\leq \varphi_L(r) \ for \ r>0$ then $\varphi(r)\equiv\underline{\varphi}_L(r)$ or $\varphi(r)\equiv \varphi_L(r) \ for \ r>0$.
\item Assume that $\varphi$ is a solution to (\ref{e}) satisfying $\varphi'(0)=0$ and $\varphi(r)\geq \varphi_L(r) \ for \ r\geq0$. Then $\varphi(r)\equiv \varphi_L(r) \ for \ r\geq0.$
\item Let $\varphi\in{S_{L_0}}$ with some $L_0\in(0,L]$. Assume that $\varphi(r)\leq \underline{\varphi}_L(r) \ for \ r\geq0$. Then $\varphi\in{S_{L_0}}$ is a minimal solution.
\item There exists no positive solution $\varphi\in{C^2(0,\infty)}$ to (\ref{e}) satisfying $\varphi(r)>\varphi_L(r) \ for  \ r\in{(0, \infty)}$ and $\varphi(r)\to\infty \ as  \ r\to0$.
\end{enumerate}
\end{proposition}

In order to prove Proposition \ref{prop7},  we prepare the following lemma.

\begin{lemma}[Naito \cite{Naito3} Proposition 4.1.]\label{lemma7}
Let $S^{(n)}_{L^{(n)}+n}$ be given by \rm{(\ref{$S_{L^{(n)}+n}$})}. Assume that there exist at least two solutions $\underline{\psi}^{(n)}_L$ and ${\psi}^{(n)}_L$ of $S^{(n)}_{L^{(n)}+n},$ where $\underline{\psi}^{(n)}_L$ is a minimal solution of $S^{(n)}_{L^{(n)}+n}$. 
\begin{enumerate}
\item If ${\psi}^{(n)}\in{S^{(n)}_{L^{(n)}+n}}$ satisfies ${\psi}^{(n)}(r)\leq {\psi}^{(n)}_L(r) \ for \ r>0$ then ${\psi}^{(n)}(r)\equiv\underline{\psi}^{(n)}_L(r)$ or ${\psi}^{(n)}(r)\equiv {\psi}^{(n)}_L(r) \ for \ r>0$.
\item Assume that ${\psi}^{(n)}$ is a solution to (\ref{psi}) satisfying ${\psi}^{(n)}(r)\geq {\psi}^{(n)}_L(r) \ for \ r\geq0$. Then ${\psi}^{(n)}(r)\equiv {\psi}^{(n)}_L(r) \ for \ r\geq0.$
\item Let ${\psi}^{(n)}\in{{S^{(n)}_{L^{(n)}_0+n}}}$ with some $L_0^{(n)}\in(0,L^{(n)}]$. Assume that ${\psi}^{(n)}(r)\leq \underline{{\psi}}^{(n)}_L(r) \ for \ r\geq0$. Then ${\psi}^{(n)}\in{S^{(n)}_{L^{(n)}_0+n}}$ is a minimal solution.
\item There exists no positive solution ${\psi}^{(n)}\in{C^2(0,\infty)}$ to (\ref{psi}) satisfying ${\psi}^{(n)}(r)>{\psi}^{(n)}_L(r) \ for  \ r\in{(0, \infty)}$ and ${\psi}^{(n)}(r)\to\infty \ as  \ r\to0$.
\end{enumerate}

\end{lemma}

\begin{proof}[Proof of Proposition \ref{prop7}]
Let  $\underline{\psi}^{(n)}_L$ and $\psi^{(n)}_L$ be given by Corollary \ref{Cor1}. \\
(i) Since $\varphi\in{S_L},$ there exists $\psi^{(n)}$ and $L^{(n)}$  by Theorem \ref{theorem4}. Since $\underline{\varphi}_L\in{S_L}$ is a minimal solution of $S_L,$ we have $\underline{\varphi}_L(r)\leq \varphi(r) \ for \ r\geq0$. Assume to the contrary that $\underline{\varphi}_L\neq \varphi$ and $\varphi\neq \varphi_L$. Then by the uniqueness of the  initial value problems to (\ref{e}), we get $\underline{\varphi}_L(r)<\varphi(r)<\varphi_L(r) \ for \ r\geq0,$ hence there exists $N\in{\mathbb{N}}$ such that $\underline{\psi}^{(n)}_L(r)<\psi^{(n)}(r)<{\psi}^{(n)}_L (r)\ for \ r\geq0, n\geq N,$ and $\psi^{(n)}\in{S^{(n)}_{L^{(n)}+n}}.$ By Lemma \ref{lemma7} (i), we have ${\psi}^{(n)}(r)\equiv\underline{\psi}^{(n)}_L(r)$ or ${\psi}^{(n)}(r)\equiv {\psi}^{(n)}_L(r) \ for \ r>0$. This is contradiction. Therefore $\varphi(r)\equiv\underline{\varphi}_L(r)$ or $\varphi(r)\equiv \varphi_L(r) \ for \ r>0$. \\
(ii) The proof is given by contradiction argument. Assume to the contrary that $\varphi\neq \varphi_L$. Then, by the uniqueness of the initial value problems to equation (\ref{e}), we have $\underline{\varphi}_L(r)<\varphi(r)<\varphi_L(r) \ \mbox{for\ all} \ r>0.$ Then there exist $N\in{\mathbb{N}}$ such that $\underline{\psi}^{(n)}_L(r)<\psi^{(n)}_L(r)<{\psi}^{(n)} (r)\ for \ r\geq0, n\geq N.$ By Lemma \ref{lemma7} (ii) we have ${\psi}^{(n)}(r)\equiv {\psi}^{(n)}_L(r) \ for \ r\geq0.$ Letting $n\to\infty,$ we obtain $\varphi(r)\equiv \varphi_L(r) \ for \ r\geq0.$ This is contradiction. Therefore $\varphi(r)\equiv \varphi_L(r) \ for \ r\geq0.$ \\
(iii) If $L_0=L$, we see that $\varphi\in{S_{L_0}}$ is a minimal solution. Let $L_0<L$. Assume to the contrary that $\varphi\in{S_{L_0}}$ is  a non-minimal solution. Then this contradicts this Proposition \ref{prop7} (ii). Therefore, $\varphi\in{S_{L_0}}$ is a minimal solution. \\
(iv) Assume to the contrary that there exists a positive solution $\varphi\in{C^2(0,\infty)}$ to (\ref{e}) satisfying the following condition:
$$
\varphi(r)>\varphi_L(r), \ r\in{(0,\infty)}, \quad \lim_{r\to0}\varphi(r)=\infty.
$$
For $\delta>0,$ let $\psi^{(n)}\in{C^2[\delta,\infty)}$ be the positive solution to initial value problem:
$$
\begin{cases}
 \displaystyle
{\psi^{(n)}}''+\biggl(\frac{N-1}{r}+\frac{r}{2}\biggr){\psi^{(n)}}'+\frac{1}{n-1}\psi^{(n)}+\biggl(\frac{\psi^{(n)}}{n}\biggr)^n=0 \\
{\psi^{(n)}}(\delta)=\varphi(\delta)+n, \quad {\psi^{(n)}}'(\delta)=\varphi'(\delta).
\end{cases}
$$
Since $\lim_{r\to0}\psi^{(n)}(r)=\lim_{r\to0}\varphi(r)+n\geq\lim_{r\to0}\varphi(r)=\infty,$ there exists a positive solution ${\psi}^{(n)}\in{C^2(0,\infty)}$ to (\ref{psi}) satisfying ${\psi}^{(n)}(r)>{\psi}^{(n)}_L(r) \ for  \ r\in{(0, \infty)}$ and ${\psi}^{(n)}(r)\to\infty \ as  \ r\to0.$ From Lemma \ref{lemma7} there exists no positive solution ${\psi}^{(n)}\in{C^2(0,\infty)}$ to (\ref{psi}) satisfying ${\psi}^{(n)}(r)>{\psi}^{(n)}_L(r) \ for  \ r\in{(0, \infty)}$ and ${\psi}^{(n)}(r)\to\infty \ as  \ r\to0$. This is  contradiction. Therefore, there exists no positive solution $\varphi\in{C^2(0,\infty)}$ to (\ref{e}) satisfying $\varphi(r)>\varphi_L(r) \ for  \ r\in{(0, \infty)}$ and $\varphi(r)\to\infty \ \mbox{as}  \ r\to0$.
\end{proof}

\section{Proof of Theorem \ref{main theorem}}

We begin this section by introducing the definition of weak supersolution and subsolution. We say that a function $u$ is  a continuous weak supersolution to (\ref{Cauchy problem}) in  $\mathbb{R}^N\times[0,T]$ if $u$ is a continuous on $\mathbb{R}^N\times[0,T],$ $u(x,0)\geq u_0(x) \ x\in{\mathbb{R}^N}$ and satisfies 
\begin{equation}\label{weak sol}
\int_{\mathbb{R}^N}u(x,t)\xi(x,t) \, dx \biggl|_{t=0}^{t=T'}\geq\int_0^{T'}\!\!\!\int_{\mathbb{R}^N}[u(x,t)(\xi_t+\Delta\xi)(x,t)+e^{u(x,t)}\xi(x,t)] \, dx \, dt,
\end{equation}
for all $T'\in{[0,T]}$ and for all $\xi\in{C^{2,1}(\mathbb{R}^N\times[0,T])}$ with $\xi\geq0$ such that $\mbox{supp}\, {\xi(\cdot,t)}$ is compact in $\mathbb{R}^N$ for all $t\in[0,T]$. A continuous weak subsolution to (\ref{Cauchy problem}) in  $\mathbb{R}^N\times[0,T]$ is defined in the same way by reversing the inequalities above.\\
We say that a function $\varphi$ is a continuous weak supersolution to (\ref{ee}) in $\mathbb{R}^N$ if $\varphi\in{C(\mathbb{R}^N)}$ satisfies
$$
\int_{\mathbb{R}^N}\biggl[\varphi\biggl(\Delta \eta -\frac{1}{2}y\cdot\nabla \eta-\frac{N}{2}\eta\biggr)+(e^{\varphi}+1)\eta\biggr] \, dy\leq0
$$
for any  $\eta\in{C^2(\mathbb{R}^N)}$ with $\eta\geq0$ such that $\mbox{supp}\, {\eta(\cdot)}$ is compact in $\mathbb{R}^N$. A continuous weak subsolution to (\ref{ee}) in  $\mathbb{R}^N$ is defined in the same way by reversing the inequalities above. \\

Next we introduce comparison principle for problem (\ref{Cauchy problem}).
\begin{lemma}[\cite{Fujishima} Lemma 2.3 (i)]
 Let $\overline{u}$ and $\underline{u}$ be continuous weak supersolution and subsolution to \rm{(\ref{Cauchy problem})} in $\mathbb{R}^N\times[0,T]$, respectively. Assume that $\overline{u}$ and $\underline{u}$ are bounded above and satisfy  $\overline{u}(x,t)-\underline{u}(x,t)\geq-Ae^{B|x|^2} \ \mbox{in} \ \mathbb{R}^N\times[0,T]$ for some constants $A,B>0$. Then $\underline{u}\leq\overline{u} \ \mbox{in} \ \mathbb{R}^N\times[0,T]$ and there exists a classical solution to (\rm{\ref{Cauchy problem}}) satisfying $\underline{u}\leq u\leq\overline{u} \ \mbox{in} \ \mathbb{R}^N\times[0,T].$ 
\end{lemma}
We show the following proposition.
\begin{proposition}\label{prop8}
 Suppose that $S_L$ have at least two elements $\underline{\varphi}_L$ and $\varphi_L$, where $\underline{\varphi}_L$ is  a minimal solution of $S_L.$
\begin{enumerate}
\item Assume that $w_0\in{C(\mathbb{R}^N)}$ satisfies $w_0(x)<\varphi_L(|x|) \ \mbox{for} \ x\in{\mathbb{R}^N}$. Then there exists a continuous weak supersolution $\overline{w}_0$ to \rm{(\ref{ee})} such that $\overline{w}_0=\overline{w}_0(r), r=|x|$ and satisfies $\overline{w}_0\not\equiv \varphi_L$ and  
\begin{equation}\label{eq91}
w_0(x)<\overline{w}_0(|x|)\leq \varphi_L(|x|), \quad x\in{\mathbb{R}^N}.
\end{equation}
\item Assume that  $w_0\in{C(\mathbb{R}^N)}$ satisfies $w_0(x)>\varphi_L(|x|) \  \mbox{for} \ x\in{\mathbb{R}^N}$. Then there exists a continuous weak subsolution $\underline{w}_0$ to \rm{(\ref{ee})} such that $\underline{w}_0=\underline{w}_0(r), r=|x|$ is  nonincreasing in $r>0$ and satisfies $\underline{w}_0\not\equiv \varphi_L$ and
\begin{equation}\label{eq91'}
\varphi_L(|x|) \leq\underline{w}_0(|x|)< w_0(x), \quad x\in{\mathbb{R}^N}.
\end{equation}
\end{enumerate}
\end{proposition}
In order to prove Proposition \ref{prop8}, we prepare the following Lemma.
\begin{lemma}\label{lemma9}
Let $\alpha_1<\alpha_2$. Assume that $\varphi(r;\alpha_i) \ (i=1,2)$ is the solution to \rm{(\ref{e})} satisfying $\varphi'(0)=0$ with initial data $\varphi(0;\alpha_i)=\alpha_i \ (i=1,2)$. Suppose that there exists $r_0>0$ such that
$$
\varphi(r;\alpha_1)<\varphi(r;\alpha_2) \ (0\leq r<r_0), \quad \varphi(r_0;\alpha_1)=\varphi(r_0;\alpha_2).
$$
If $\alpha_3>\alpha_2$, then $\varphi(r;\alpha_3)-\varphi(r;\alpha_2)$ has at least one zero in $(0,r_0).$
\end{lemma}
\begin{proof}
This proof is carried out by the similar argument used in the proof of [\cite{Naito3} Lemma 5.1]. Assume to the contrary that $\varphi(r;\alpha_3)-\varphi(r;\alpha_2)>0, \ \mbox{for} \ 0\leq r<r_0.$ We set $\phi_1(r)=\varphi(r;\alpha_2)-\varphi(r;\alpha_1), \ \phi_2(r)=\varphi(r;\alpha_3)-\varphi(r;\alpha_2)$. Since $\varphi(r;\alpha_i) \ (i=1,2,3)$ is the solution to (\ref{e}) we have
\begin{equation}\label{eq92}
(\rho_N{\phi}'_j)'+\rho_Nm_j\phi_j=0 \quad \mbox{for} \ r>0, \ j=1,2,
\end{equation}
where $\rho_N(r)=r^{N-1}e^{r^2/4}$ and $m_j$ satisfies:
$$
e^{\varphi(r;\alpha_i)}<m_j(r)<e^{\varphi(r;\alpha_{j+1})} \quad 0\leq r\leq r_0, \ j=1,2.
$$
Then, we obtain $m_1(r)<m_2(r) \  \mbox{for} \ 0\leq r<r_0$ and
\begin{equation}\label{eq92'}
\phi_1'(r_0)\leq0, \ \phi_2(r_0)\geq0.
\end{equation}
By (\ref{eq92}) we have
\begin{eqnarray}\label{eq93}
(\rho_N{\phi}'_1)'\phi_2+\rho_Nm_1\phi_1\phi_2&=&0 \quad (r>0), \\
\label{eq94}
 (\rho_N{\phi}'_2)'\phi_1+\rho_Nm_2\phi_1\phi_2&=&0 \quad (r>0).
\end{eqnarray} 
Since (\ref{eq93}) and (\ref{eq94}),  we have
\begin{equation}\label{eq95}
(\rho_N(\phi'_1\phi_2-\phi_1\phi'_2))'=-\rho_N(m_1-m_2)\phi_1\phi_2.
\end{equation}
We integrate (\ref{eq95}) from $0$ to $r_0$, we obtain
$$
\rho_N(\phi'_1\phi_2-\phi_1\phi'_2)\mid^{r=r_0}_{r=0}=-\int_0^{r_0}\rho_N(m_1-m_2)\phi_1\phi_2>0.
$$
On the other hand, since (\ref{eq92'}) and $\phi_i(0)'=0$ we have
$$
\rho_N(\phi'_1\phi_2-\phi_1\phi'_2)\mid^{r=r_0}_{r=0}=\rho_N(r_0)\phi'_1(r_0)\phi_2(r_0)\leq0.
$$
This is contradiction. Therefore $\varphi(r;\alpha_3)-\varphi(r;\alpha_2)$ has at least one zero in $(0,r_0)$.
\end{proof}
\begin{lemma}\label{lemma10}
Assume that $S_L$ has at least two elements $\underline{\varphi}_L$ and  $\varphi_L$. Suppose that $\alpha_*=\underline{\varphi}_L(0), \alpha^*=\varphi_L(0)$ and $\alpha_0\in{(\alpha_*,\alpha^*)}$. Then there exists $r_0>0$ such that
\begin{equation}\label{eq96}
\underline{\varphi}_L(r)<\varphi(r;\alpha_0)<\varphi_L(r) \ \mbox{for} \ 0\leq r<r_0, \quad \varphi(r_0;\alpha_0)=\varphi_L(r_0).
\end{equation}
In addition, we have
\begin{enumerate}
\item If $\alpha\in{(\alpha_0,\alpha^*)}$ then there exists $r_1\in{(0,r_0)}$ such that
\begin{equation}\label{eq97}
\varphi(r;\alpha)<\varphi_L(r) \ (0\leq r<r_1), \quad \varphi(r_1;\alpha)=\varphi_L(r_1);
\end{equation}
\item If $\alpha>\alpha^*$ then there exists $r_2\in{(0,r_0)}$ such that
\begin{equation}\label{eq98}
\varphi(r;\alpha)>\varphi_L(r) \ (0\leq r<r_2), \quad \varphi(r_2;\alpha)=\varphi_L(r_2).
\end{equation}
\end{enumerate}
\end{lemma}

\begin{proof}
Since we see that $\underline{\varphi}_L(0)<\varphi(0;\alpha_0)<\varphi_L(0),$ one of the following condition (a)-(c) holds: \\
(a) $\underline{\varphi}_L(r)<\varphi(r;\alpha)<\varphi_L(r) \ r>0;$ \\
(b) There exists $r_0>0$ such that
$$
\underline{\varphi}_L(r)<\varphi(r;\alpha_0)<\varphi_L(r) \ 0\leq r<r_0 \quad \underline{\varphi}_L(r)=\varphi(r_0;\alpha_0);
$$
(c) There exists $r_0>0$ satisfying (\ref{eq96}). \\
The condition (a) does not hold by Proposition \ref{prop7} (i). Assume that condition (b) holds. By Lemma \ref{lemma9}, $\varphi_L(r)-\varphi(r;\alpha_0)$ has at least one zero in $(0,r_0)$. This is contradiction. Therefore the condition (c) holds, and we have (\ref{eq96}).
\begin{enumerate}
\item Let $\alpha\in{(\alpha_0,\alpha^*)}$. Assume that $\varphi(r;\alpha)<\varphi_L(r) \ \mbox{for} \ 0\leq r<r_1.$ By (\ref{eq96}), there exists $r_1\in{(0,r_0]}$ such that 
$$
\varphi(r;\alpha_0)<\varphi(r;\alpha) \ (0\leq r<r_1),\quad \varphi(r_1;\alpha_0)=\varphi(r_1;\alpha).
$$ 
By Lemma \ref{lemma9}, $\varphi_L(r)-\varphi(r;\alpha)$ has at least one zero in $(0,r_0)$.
This is contradiction. We have  (\ref{eq97}).
\item Let $\alpha_1=\alpha_0, \alpha_2=\alpha^*$ and $\alpha_3=\alpha$. By Lemma \ref{lemma9}, $\varphi(r;\alpha)-\varphi_L(r)$ has at least one zero in $(0,r_0)$. Therefore (\ref{eq98}) holds.
\end{enumerate}
\end{proof}
\begin{lemma}[\cite{Fujishima} Lemma 2.5]\label{lemma11}
\begin{enumerate}
\item Let $\varphi_1=\varphi_1(|y|)$ and $\varphi_2=\varphi_2(|y|)$ be radially symmetric subsolutions to \rm{(\ref{ee})}. Assume that there exists $R>0$ such that $\varphi_1(R)=\varphi_2(R)$ and $\varphi'_1(R)\leq\varphi'_2(R)$. Then, $\underline{\varphi}$ defined by 
$$
\underline{\phi}(r):=\begin{cases}
\varphi_1(r), & r\in{[0,R]}, \\
\varphi_2(r) & r\in{[R,\infty)}
\end{cases}
$$
is a continuous weak subsolution to (\ref{ee}).
\item Let $\varphi_1=\varphi_1(|y|)$ and $\varphi_2=\varphi_2(|y|)$ be radially symmetric  supersolutions to \rm{(\ref{ee})}. Assume that there exists $R>0$ such that $\varphi_1(R)=\varphi_2(R)$ and $\varphi'_1(R)\geq\varphi'_2(R)$. Then $\overline{\phi}$ defined by 
$$
\overline{\phi}(r):=\begin{cases}
\varphi_1(r), & r\in{[0,R]}, \\
\varphi_2(r) & r\in{[R,\infty)}
\end{cases}
$$
is a continuous  weak supersolution to (\ref{ee}).
\end{enumerate}
\end{lemma}
\begin{proof}[Proof of Proposition \ref{prop8}]
Let $\alpha_*=\varphi_L(0), \alpha^*=\varphi_L(0)$ and $\alpha_0\in{(\alpha_*,\alpha^*)}$. By Lemma \ref{lemma10}, there exists $r_0>0$ satisfying (\ref{eq96}). 
\begin{enumerate}
\item Put $w_M(r)=\max_{|x|=r}w_0(x) \ \mbox{for} \ r>0$. Then we have $\varphi_L(r)>w_M(r) \ \mbox{for} \ r>0$. Setting $\varepsilon=\min_{0\leq r\leq r_0}|\varphi_L(r)-w_M(r)|$. By continuous dependence of initial data, there exists $\delta>0$ such that if $|\alpha-\alpha^*|<\delta$ then 
\begin{equation}\label{eq99}
|\varphi_L(r)-\varphi(r;\alpha)|<\varepsilon \quad  \mbox{for} \ 0\leq r\leq r_0.
\end{equation}
Let $\alpha\in{(\alpha^*-\delta,\alpha^*)\cap(\alpha_0,\alpha^*)}$. By Lemma \ref{lemma10} (i), there exists $r_0\in{(0,r_1)}$ such that (\ref{eq98}). Then we have 
$$
w_M(r)\leq \varphi_L(r)-\varepsilon<\varphi(r;\alpha)<\varphi_L(r) \quad \mbox{for} \ 0\leq r\leq r_1
$$
and $\varphi(r_1;\alpha)=\varphi_L(r_1)$. Therefore we obtain $\varphi'(r_1;\alpha)\geq v'_L(r_1)$.
Putting
$$
\overline{w}_0(r)=\begin{cases}
\varphi(r;\alpha), & 0\leq r<r_1, \\
\varphi_L(r), & r\geq r_1.
\end{cases}
$$
Then $\overline{w}_0$ satisfies (\ref{eq91}) and we have $\overline{w}_0$ is a continuous weak supersolution to (\ref{e}) by Lemma \ref{lemma11} (ii).
\item Put $w_m(r)=\min_{|x|=r}w_0(x) \ \mbox{for} \ r>0$. Then we have $\varphi_L(r)<w_m(r) \ \mbox{for} \ r>0$. Setting $\varepsilon=\min_{0\leq r\leq r_0}|\varphi_L(r)-w_m(r)|$. By the continuous dependence of initial data, there exists $\delta>0$ such that if $|\alpha-\alpha^*|<\delta$ then 
\begin{equation}\label{eq99}
|\varphi_L(r)-\varphi(r;\alpha)|<\varepsilon \quad \mbox{for} \ 0\leq r\leq r_0.
\end{equation}
Put $\alpha\in{(\alpha^*,\alpha^*+\delta)}$. By Lemma \ref{lemma10} (ii), there exists $r_2\in{(0,r_0)}$ satisfying (\ref{eq98}). Then we have 
$$
\varphi_L(r)<\varphi(r;\alpha)<\varphi_L(r)+\varepsilon\leq w_m(r) \quad \mbox{for} \ 0\leq r<r_2
$$
and $\varphi(r_2;\alpha)=\varphi_L(r_2)$. Therefore we have $\varphi'(r_2;\alpha)\leq \varphi'_L(r_2)$.
Put 
$$
\underline{w}_0(r)=\begin{cases}
\varphi(r;\alpha), & 0\leq r<r_2, \\ 
\varphi_L(r), & r\geq r_2.
\end{cases}
$$
Then $\underline{w}_0$ satisfies (\ref{eq91'}) and  $\underline{w}'_0(r)\leq0 \ \mbox{for} \ r\geq0$. We obtain that $\underline{w}_0$ is a continuous weak subsolution to (\ref{e}).
\end{enumerate}
\end{proof}
In order to prove Theorem \ref{main theorem}, we use the self-similar variables. 
Let $u$ be the solution to (\ref{Cauchy problem}). Then we define $w$ by the following:
\begin{equation}\label{eq100}
w(y,s):=\log{(1+t)}+u(x,t), \quad y=\frac{x}{\sqrt{1+t}}, \quad s=\log{(1+t)}.
\end{equation}
Then, $w$ satisfy
\begin{equation}\label{eq101}
\begin{cases}\displaystyle
w_s=\Delta w+\frac{1}{2}y\cdot\nabla w+e^{w}+1 & \mbox{in}  \ \mathbb{R}^N\times(0,\infty), \\
w(y,0)=w_0(y) & \mbox{on} \ \mathbb{R}^N
\end{cases}
\end{equation}
where $w_0=u_0$. \\

We say that $w$ is a continuous weak supersolution to (\ref{eq101}) in $0\leq s\leq S$ if $w$ is a continuous on $\mathbb{R}^N\times[0,S]$, $w(y,0)\geq w_0(y) \ y\in{\mathbb{R}^N}$ and satisfies
\begin{equation}\label{weak sol}
\int_{\mathbb{R}^N}w(y,s)\xi(y,s) \, dy \biggl|_{s=0}^{s=\sigma}\geq\int_0^{\sigma}\!\!\!\int_{\mathbb{R}^N}[w(y,s)(\xi_s+\Delta\xi)(y,s)+e^{w(y,s)}\xi(y,s)] \, dy \, ds
\end{equation}
for all $\xi\in{C^{2,1}(\mathbb{R}^N\times[0,S])}$ with  $\xi\geq0$ such that $\mbox{supp} \, \xi(\cdot,s)$ is compact in $\mathbb{R}^N$ for all $s\in{[0,\sigma]}$. A continuous weak subsolution is defined in the same way by reversing the inequalities. \\

Next we introduce comparison results of sub- and supersolutions..
\begin{lemma}[\cite{Fujishima} Lemma 2.3 (ii)]\label{lemma14}
 Let $\overline{w}$ and $\underline{w}$ be weak supersolution and subsolution to \rm{(\ref{eq101}}) in $\mathbb{R}^N\times[0,S]$, respectively. Assume that $\overline{w}$ and $\underline{w}$ are bounded above and satisfy $\overline{w}(x,s)-\underline{w}(x,s)\geq-Ae^{B|x|^2} \ \mbox{in} \ \mathbb{R}^N\times[0,S]$ for some constants $A,B>0.$ Then $\underline{w}\leq\overline{w} \ \mbox{in} \ \mathbb{R}^N\times[0,S]$ and there exists the solution $w$ to \rm{(\ref{eq101})} such that $\underline{w}\leq w\leq\overline{w} \ \mbox{in} \ \mathbb{R}^N\times[0,S]$.
\end{lemma}
We will prove the following proposition.
\begin{proposition}\label{prop9}
Let  $3\leq N\leq9$. Assume that there exist at least two elements $\underline{\varphi}_L$ and $\varphi_L$ of $S_L,$ where $\underline{\varphi}_L$ is a minimal solution of $S_L.$ Suppose that $w_0$ holds the assumption \rm{(\ref{assumption})}.
\begin{enumerate}
\item If $w_0(y)>\varphi_L(|y|) \ \mbox{for} \ y\in{\mathbb{R}^N}$ then the solution $w$ to \rm{(\ref{eq101})} with initial data $w_0$ blows up in finite time.
\item If $w_0(y)<\varphi_L(|y|) \ \mbox{for} \ y\in{\mathbb{R}^N}$
then the solution $w$ to \rm{(\ref{eq101})} with initial data $w_0$  exists globally in time. 
\end{enumerate}
\end{proposition}
To prove Proposition \ref{prop9} we prepare the following Lemmas.

\begin{lemma}[\cite{Fujishima} Lemma 2.4.]\label{lemma12}
\begin{enumerate}
\item Let $w_0$ be continuous weak subsolution to \rm{(\ref{ee})}. Assume that the solution $w$ to \rm{(\ref{eq101})} with initial data $w_0$ exists globally in time. Then $w$ is nondecreasing in $s$.
\item Let $w_0$ be continuous weak supersolution to \rm{(\ref{ee})}. Assume that the solution $w$ to \rm{(\ref{eq101})} with initial data $w_0$ exists globally in time. Then $w$ is nonincreasing in $s$.
\end{enumerate}
\end{lemma}

\begin{lemma}[\cite{Fujishima} Lemma 2.7]\label{lemma13}
Let the solution $w=w(|y|,s)$ to \rm{(\ref{eq101})} be a global solution and radially symmetric  in $y$ . Assume that  $w(|y|,s)$ is nondecreasing function in $s$ for each fixed $r\geq0$ and nonincreasing function in $r=|y|$ for each fixed $s\geq0$. Put  $\varphi(r):=\lim_{s\to\infty}w(r,s)$.
\begin{enumerate}
\item If $\varphi$ is bounded above, then $\varphi\in{C^2([0,\infty))}$ is the solution to (\ref{e}) satisfying $\varphi'(0)=0$.
\item If $\varphi$ is not  bounded above. Then  $\varphi\in{C^2((0,\infty))}$ is the solution to (\ref{e}) satisfying $\lim_{r\to0}\varphi(r)=\infty$.
\end{enumerate}
\end{lemma}

\begin{proof}[Proof of \rm{ Proposition \rm{\ref{prop9}}}]
(i) The proof is carried out  contradiction argument. Assume to contrary that $w$ exists globally in time. By Proposition \ref{prop8} (ii) there exists a continuous weak subsolution $\underline{w}_0$ such that $\underline{w}_0=\underline{w}_0(r), \ r=|x|,$ nonincreasing in $r$, $\underline{w}_0\neq \varphi_L$ and (\ref{eq91'}). Let  $\underline{w}$ be the solution to (\ref{eq101}) with initial data $w_0=\underline{w}_0$. From Lemma \ref{lemma14},  $\underline{w}=\underline{w}(r,s), \ r=|y|$ is a radially symmetric  and nonincreasing function in $r\geq0$. We remark that  $w_0=u_0$ satisfies assumption (\ref{assumption}). By the comparison principle, we have $\varphi_L<\underline{w}< w$ and  $\underline{w}$ exists globally in time. From Lemma \ref{lemma12}, $\underline{w}(r,s)$ is  nonincreasing in $s.$ Let $\underline{\varphi}(r)=\lim_{s\to\infty}\underline{w}(r,s) \ \mbox{for} \ r\geq0$. Since $\underline{w}(r,s)$ is nonincreasing in $r,$ $\underline{\varphi}(r)$ is a nonincreasing function and satisfies
\begin{equation}\label{eq104}
\varphi_L(r)<\underline{w}_0(r)\leq \underline{w}(r,s)\leq\underline{\varphi}(r), \quad r>0, \ s\geq0.
\end{equation}
Assume that $\underline{\varphi}$ is bounded above. By Lemma \ref{lemma13} (i), we get $\underline{\varphi}\in{C^2[0,\infty)}$ to (\ref{e}) satisfying $\underline{\varphi}'(0)=0$. From (\ref{eq104}), we have  $\varphi_L(r)<\underline{\varphi}(r) \ \mbox{for} \ r\geq0.$ This is a contradiction by Proposition \ref{prop7} (ii). Therefore we conclude that $\underline{\varphi}\not\in{L^{\infty}[0,\infty)}$. By Lemma \ref{lemma13} (ii), we have $\underline{\varphi}\in{C^2(0,\infty)}$ satisfying $\lim_{r\to0}\underline{\varphi}(r)=\infty$. From (\ref{eq104}), we obtain $\varphi_L(r)<\underline{\varphi}(r) \ \mbox{for} \ 0<r<\infty$. This is a contradiction by Proposition \ref{prop7} (iv). Therefore the solution $w$ to (\ref{eq101}) with initial data $w_0$ blows up in finite time. \\
(ii) Since $\varphi_L$ is the stationary solution to (\ref{ee}) and $w_0$ satisfies  assumption (\ref{assumption}), we obtain 
$$
w(y,s)\leq \varphi_L(|y|), \quad y\in{\mathbb{R}^N}, \ s>0
$$
by the comparison principle. Therefore the solution $w$ to (\ref{eq101}) exists globally in time.
\end{proof}

\begin{proof}[Proof of \rm{Theorem \ref{main theorem}}]
(i) The proof is carried out by contradiction argument. Assume to contrary that $u$ exists globally in time. By the comparison principle, $u(x,t)>u_L(x,t_0+t) \   x\in{\mathbb{R}^N}, \ t>0.$ Hence assume that $u_0(x)>u_L(x,t_0) \ \mbox{for} \ x\in{\mathbb{R}^N}$. Then we have
\begin{equation}\label{thm5.1}
\log{t_0}+u_0(\sqrt{t_0}x)>\varphi_L(|x|), \quad x\in{\mathbb{R}^N}.
\end{equation}
Let ${w}(x,t)=\log{t_0}+u(\sqrt{t_0}x,t_0t)$. Then ${w}$ satisfies the following:
\begin{equation}\label{thm5.2}
{w}_t=\Delta{{w}}+e^{{w}} \quad \mbox{in} \ \mathbb{R}^N\times(0,\infty), \quad {w}(x,0)=\log{t_0}+u_0(\sqrt{t_0}x) \quad \mbox{in} \ \mathbb{R}^N.
\end{equation}
Put $$\hat{w}(y,s)=\log{(t+1)}+{w}(x,t), \quad y=\frac{x}{\sqrt{1+t}}, \quad s=\log{(1+t)}.
$$ 
Then $\hat{w}$ is a global solution to (\ref{eq101}) satisfying $\hat{w}_0(y)={w}(y,0) \ \mbox{for} \ y\in{\mathbb{R}^N}.$ From  (\ref{thm5.1}) and (\ref{thm5.2}), we have $\hat{w}_0(y)>\varphi_L(|y|) \ \mbox{for} \ y\in{\mathbb{R}^N}.$ By proposition \ref{prop9} (i), The solution $\hat{w}$ to (\ref{eq101}) blows up in finite time. This is contradiction. Therefore $u$ blows up in finite time. \\
(ii) Assume that $u_0(x)<u_L(x,t_0) \ \mbox{for} \ y\in{\mathbb{R}^N}$ by the comparison principle. Then we have
\begin{equation}\label{thm5.2-1}
\log{t_0}+u_0(\sqrt{t_0}x)<\varphi_L(|x|) \quad \mbox{for} \ x\in{\mathbb{R}^N}.
\end{equation}
Put $w(x,t)=\log{t_0}+u(\sqrt{t_0}x,t_0t)$. Then ${w}$ satisfies (\ref{thm5.2}).
Let
$$
\hat{w}(y,s)=\log{(t+1)}+{w}(x,t), \quad y=\frac{x}{\sqrt{1+t}}, \quad s=\log{(1+t)}.
$$
Then $\hat{w}$ is the solution to (\ref{eq101}) satisfying $\hat{w}_0(y)={w}(y,0) \ \mbox{for} \ y\in{\mathbb{R}^N}.$ From (\ref{thm5.2-1}), we obtain $w_0(y)<\varphi_L(|y|) \  \mbox{for} \ y\in{\mathbb{R}^N}.$
By proposition \ref{prop9} (ii), $\hat{w}$ exists globally. Therefore $w$ exists globally.
\end{proof}

\section*{Acknowledgment}
The author would like to express his gratitude to professor Mitsuru Sugimoto for valuable advice and continuous encouragement. He would like to offer my special thanks to Professor Yuki Naito for his helpful comment. He has had the support and continuous encouragement. The author would like to thank Sugimoto's office members for useful discussions. The author also thank Mr. Soichiro Suzuki in Sugimoto's office members for supporting my research activities.

\end{document}